\documentclass{amsart}

\usepackage{amssymb, amsmath, latexsym, amsfonts, amsthm, mathrsfs} 
\usepackage{dsfont}
\usepackage{mathabx}
\usepackage{amsrefs} 
\usepackage{verbatim} 
\usepackage[text={6in,9in},centering]{geometry} 
\usepackage[colorlinks=true]{hyperref} 
\usepackage[doublespacing]{setspace} 
\usepackage{pgfplots}
\usepackage{graphicx}
\usepackage{subfig}
\usepackage{multirow}
\usepackage{array}
\usepackage{graphicx}
\usepackage{tikz}
\usepackage{pgfplots}
\usepackage{verbatim}
\usepackage{algpseudocode}
\usepackage{bm}
\usepackage{bbm}
\usepackage{dsfont}
\usepackage{enumitem}
\usepackage{bm}      
\usepackage{pgfplots}
\usepackage{pgfplotstable}
\usepgfplotslibrary{patchplots}
\usepgfplotslibrary{colormaps}
\pgfplotsset{compat=newest}
\usepackage{tikz}
\usepackage{algorithm}
\usepackage{caption}
\usepackage[utf8]{inputenc}
\usepackage{algorithm,algcompatible,amssymb,amsmath}

\algnewcommand\algorithmicto{\textbf{to}}
\algnewcommand\RETURN{\State \textbf{return} }
\usepackage[toc,page]{appendix}
\newtheorem{thm}{\bf Theorem}[section]

\newtheorem{lem}[thm]{\bf Lemma}
\newtheorem{prop}[thm]{\bf Proposition}
\newtheorem{defn}[thm]{\bf Definition}
\newtheorem{rem}[thm]{\bf Remark}
\newtheorem{exam}[thm]{\bf Example}

\usepackage{amsmath,amssymb}
\usepackage[bb=boondox]{mathalfa}
\numberwithin{equation}{section}
\def\zR{\mathbb R}

\setlist{nosep} 
\begin{document}
	\def\leftmark{ }
	\renewcommand{\thefootnote}{}
	\title[ STOCHASTIC MATRICES AND MAJORIZATION IN MAX ALGEBRA ]{ STOCHASTIC MATRICES AND MAJORIZATION IN MAX ALGEBRA}


	\maketitle
	
	\authors{S. M. Manjegani, T. Parsa}
	
	\begin{abstract}

In this paper, we introduce and characterize max-doubly stochastic matrices within the framework of max algebra, where the operations are defined as $x \oplus y = \max(x, y)$ and $x \otimes y = xy$. We explore the fundamental properties of max-doubly stochastic matrices and their role in vector majorization. Specifically, we establish that for vectors $x$ and $y$ in max algebra, $x$ is majorized by $y$ if there exists a max-doubly stochastic matrix $D$ such that $x = D \otimes y$. This provides a new approach to majorization theory within tropical mathematics and enhances the understanding of vector relations in max algebra.
	\end{abstract}
	\keywordsname:{ max algebra, max-doubly stochastic matrix, max-majorization}
	
	\maketitle

	
	\section{\textbf{Introduction}}

Majorization is a foundational concept in the study of inequalities within the space of real vectors, with its origins rooted in the pioneering works of Muirhead, Lorenz, Dalton, Schur, Hardy, Littlewood, and Polya \cite{arnold, marshal}. The theory of majorization has far-reaching applications across virtually most areas of mathematics, making it a critical tool for understanding and comparing vector relations \cite{marshal, nielsen }.

Simultaneously, tropical geometry has emerged as a powerful framework in various mathematical applications, including scheduling, optimal control, discrete event systems, and statistical physics. Central to tropical geometry is the max-plus semiring, a semiring defined over $\zR\cup\{-\infty\}$ with tropical addition ($\oplus$) and tropical multiplication ($\otimes$). These operations are defined as $x\oplus y=\max(x,y)$ and $x\otimes y=x+y$, with identity elements of $-\infty$ for addition and $0$ for multiplication \cite{butkovich}.  We denote this semiring by $\zR_{\max,+}$.

This paper focuses on an isomorphic variation of the max-plus algebra known as max algebra. In max algebra, the semiring is defined as ($\zR_{\geq 0}, ~\oplus, ~\otimes$), where the operations are $x \oplus y = \max(x, y)$ and $x \otimes y = xy$ \cite{butkovich}. We denote this semiring by $\zR_{\max}$.

These operations extend naturally to matrices and vectors. Given matrices $A = (a_{i,j})$, $B = (b_{i,j})$ and $C = (c_{i,j})$ of compatible sizes with entries from $\zR_{\max}$ and a scalar $\alpha \in \zR_{\max}$, matrix operations are defined as follows:
$C = A \oplus B$ if $c_{i,j} = a_{i,j} \oplus b_{i,j}$ for all $i, j$;
$C = A \otimes B$ if $c_{i,j} = \bigoplus_k \left(a_{i,k} \otimes b_{k,j}\right) = \max_k (a_{i,k} b_{k,j})$ for all $i, j$ and
$C = \alpha \otimes A=\alpha A$ if $c_{i,j} = \alpha a_{i,j}$ for all $i, j$ \cite{butkovich}.
Since the operations of max-addition and scalar multiplication are well-defined for matrices of compatible sizes, we will consistently apply this interpretation to the set of $n \times n$ matrices throughout the discussion. We use the symbols $ M_{n}\left( \zR _{\max} \right) $ and $ \zR _{\max}^{n} $, respectively, to represent the set of $ n\times n $ matrices with entries from $ \zR_{\max} $ and the set of vectors defined on max algebra.

The objective of this paper is to introduce and study the notion of vector majorization within the max algebra framework. In classical majorization theory, a vector $x$ is said to be majorized by $y$ (denoted $x \prec y$) if and only if there exists a doubly stochastic matrix $D$ such that $x = Dy$ \cite{ando}. This work extends the classical concept by exploring majorization with respect to max-doubly stochastic matrices, offering new insights into the structure and properties of vector majorization in max algebra.
\section{\textbf{Stochastic Matrices in Max Algebra}}
In this section, we define max-doubly stochastic matrices and explore their key properties. We will demonstrate that the set of all max-doubly stochastic matrices forms a max-convex set. Furthermore, we will characterize the max-extreme points of this set.
We denote the vectors $\left( 1~1~\dots~ 1\right)^T $ and $\left( 0~0~\dots ~0\right)^T\in \zR_{\max}^{n} $ by $ \mathbb{1}$ and $\mathbb{O} $, respectively and $ e_i$ is a vector in $ \zR_{\max}^{n} $, whose $ i $-th entry is $ 1 $ and the rest of its entries are zero.
\begin{defn}\label{def2.1}
		Suppose that $D\in  M_{n}\left( \zR _{\max} \right) $. 
		\begin{enumerate}
			\item\label{def3.1}
			$ D $ is a max-row stochastic matrix if $ D\otimes \mathbb{1}=\mathbb{1} $.
			\item\label{def3.2}
			$ D $ is a max-column stochastic matrix if $\mathbb{1}^T  \otimes D =\mathbb{1}^T$.
			\item
			$ D $ is a max-doubly stochastic matrix if $ D $ is a max-row stochastic matrix and a max-column stochastic matrix.
	\end{enumerate}
	\end{defn}
\begin{exam}
There are some examples of max-doubly stochastic matrix.
\begin{align*}
	I= \begin{pmatrix}
		1 & 0  & 0 \\
		0 & 1 & 0 \\
		0 & 0 & 1
	\end{pmatrix},~~
	P= \begin{pmatrix}
		0 & 1  & 0 \\
		1 & 0 & 0 \\
		0 & 0 & 1
	\end{pmatrix},~~
	D_{1}= \begin{pmatrix}
		\frac{1}{2} & \frac{1}{4}  & 1 \\
		\frac{4}{5} & 1 & \frac{2}{3} \\
		1 & \frac{2 }{3}& \frac{6}{7}
	\end{pmatrix} ~~\mbox{and}~~
	D_{2}= \begin{pmatrix}
		1 & 1  & 0 \\
		0 & 1 & 1 \\
		0 & 0 & 1
	\end{pmatrix}.
\end{align*}
In fact, every row and column of max-doubly stochastic matrices has at least one entry equal to 1.
\end{exam}
We will denote the set of all  $ n\times n $ max-doubly stochastic matrices by $ \bm{MDS_n\left(\mathbb{R}_{\max}\right)}$.
It is interesting to note that if $D_1$ and $D_2$ belong to the set of max-doubly stochastic matrices, then both $D_1 \oplus D_2$ and $D_1 \otimes D_2$ also belong to $MDS_n(\mathbb{R}_{\max})$. 
\begin{thm}\label{thm2.2}
	For \( D_1, D_2 \in MDS_n\left(\mathbb{R}_{\max}\right) \), both \( D_1 \oplus D_2 \) and \( D_1 \otimes D_2 \) are max-doubly stochastic matrices.
	\begin{proof}
		Let \( D_1, D_2 \in MDS_n\left(\mathbb{R}_{\max}\right) \). Then,
		\[
		\left(D_1 \oplus D_2\right) \otimes \mathbb{1} = \left(D_1 \otimes \mathbb{1}\right) \oplus \left(D_2 \otimes \mathbb{1}\right) = \mathbb{1} \oplus \mathbb{1} = \mathbb{1},
		\]
		and
		\[
		\left(D_1 \otimes D_2\right) \otimes \mathbb{1} = D_1 \otimes \left(D_2 \otimes \mathbb{1}\right) = D_1 \otimes \mathbb{1} = \mathbb{1}.
		\]
		This implies that \( D_1 \oplus D_2 \) and \( D_1 \otimes D_2 \) are max-row stochastic matrices. Similarly proves that \( D_1 \oplus D_2 \) and \( D_1 \otimes D_2 \) are max-column stochastic matrices.
	\end{proof}
\end{thm}
Because of being isomorphic $\mathbb{R}_{\max,+}$ and $\mathbb{R}_{\max}$, we can refer to relation $ \leq $ and its properties on max algebra. For $a, b \in \mathbb{R}_{\max}$, $a \leq b$ holds, whenever $a \oplus b = b$.
For vectors $x, y \in\zR _{\max}^{n}$,  $x \leq y$ if and only if $x_i \leq y_i$ for all $1 \leq i \leq n$. Similarly; for matrices $ S,D\in M_{n}\left( \zR _{\max} \right) $, the relation $ S\leq D $ is defined in an analogous manner \cite{butkovich}.
\begin{thm}\label{theorem6}
	Let \( D_1, D_2 \in MDS_n\left(\mathbb{R}_{\max}\right) \). For any matrix \( D \in M_n\left( \mathbb{R}_{\max} \right) \), if \( D_1 \leq D \leq D_2 \), then \( D \) is a max-doubly stochastic matrix.
	\begin{proof}
		By \cite[Lemma 1.1.1]{butkovich} and the isomorphism between \( \mathbb{R}_{\max} \) and \( \mathbb{R}_{\max,+} \), we have
		\[
		\mathbb{1} = D_1 \otimes \mathbb{1} \leq D \otimes \mathbb{1} \leq D_2 \otimes \mathbb{1} = \mathbb{1}.
		\]
		Thus, \( D \otimes \mathbb{1} = \mathbb{1} \), which shows that \( D \) is a max-row stochastic matrix. A similar argument proves that \( D \) is a max-column stochastic matrix, too.
	\end{proof}
\end{thm}
\begin{defn}\label{def2.2}
		Let $A\in  M_{n}\left( \zR _{\max} \right) $.
		\begin{enumerate}
			\item 
				$ A $ is max-trace preserving if $ \text{tr}\left( A\otimes x\right)=\text{tr} \left( x \right) $ for all  $ x\in\zR _{\max}^{n} $, where $ \text{tr} \left( x \right)=\max_{i\leq 1\leq n} x_i $.
			\item 
				$ A $ is max-unital preserving if $ A\otimes\mathbb{1}=\mathbb{1}$.
		\end{enumerate}	

\end{defn}
We can see in the following theorem the relationship between max-doubly stochastic matrices and Definition \ref{def2.2}.
\begin{thm}
	$D \in M_{n}\left( \zR _{\max} \right)$ is a max-doubly stochastic matrix if and only if $D$ is max-unital preserving and max-trace preserving.
	\begin{proof}
		Let $D \in M_{n}\left( \zR _{\max} \right)$ be a max-doubly stochastic matrix, and let $x = \left(x_1, x_2, \dots, x_n\right)^T \in \zR_{\max}^{n}$. By Definition \ref{def2.1}, since $D \otimes \mathbb{1} = \mathbb{1}$, it follows that the matrix $D$ is max-unital preserving. 
		
		For max-trace preservation, note that $\text{tr}(x) = \mathbb{1}^T \otimes x$. Thus,
		\[
		\text{tr}(D \otimes x) = \mathbb{1}^T \otimes \left(D \otimes x\right) = \left(\mathbb{1}^T \otimes D\right) \otimes x = \mathbb{1}^T \otimes x = \text{tr}(x).
		\]
		Hence, the matrix $D$ is also max-trace preserving.
		
		Conversely, suppose the matrix $D$ is max-unital preserving and max-trace preserving. Since $D$ is max-unital preserving, it is a max-row stochastic matrix. We now show that $D$ is also a max-column stochastic matrix.
		Because $D$ is a max-trace preserving matrix, for each $1 \leq i \leq n$, we have
		\[
		\bigoplus_{j=1}^{n} d_{j,i} = \text{tr}\left(D \otimes e_i\right) = \text{tr}\left(e_i\right) = 1.
		\]	
		Therefore, $\mathbb{1}^T \otimes D = \mathbb{1}^T$, which shows that $D$ is a max-column stochastic matrix.
	\end{proof}
\end{thm}
Recall from \cite{butkovich} that an \textbf{eigenvector} of a matrix $A \in M_n(\mathbb{R}_{\max})$ is a nonzero vector $x \in \mathbb{R}_{\max}^n$ that satisfies $A \otimes x = \lambda \otimes x$ for some scalar $\lambda \in \mathbb{R}_{\max}$, called the \textbf{eigenvalue} of the matrix $A$.
The \textbf{spectrum} of a matrix $A\in M_n(\zR_{\max})$ in max algebra, denoted by $\mathbf{\sigma}_{\otimes}(A)$, is defined as
\[
\sigma_{\otimes}(A)=\{\lambda\ge 0~|  ~\mbox{there exists}~ 0\neq x\in\zR_{\max}^{n}~\mbox{ with}~
A\otimes x=\lambda x\}.
\]
The \textbf{spectral radius} of $A\in M_n(\zR_{\max})$ represents the maximum cycle geometric mean of $A$ and is defined by:
\begin{align}\label{rabete1.1}
	r_{\otimes}\left( A \right) &=\max \left\{  \left(  a_{i_{1},i_{2}} a_{i_2,i_3}\cdots a_{i_k, i_1}\right)^{\frac{1}{k}} ~| ~ 1\leq k\leq n,~ i_1,\dots ,i_k\in \left\{1,\dots ,n \right\}~ \mbox{are distinct}\right\}\\
	&= \max \left\{ \lambda~| ~\lambda\in \sigma_{\otimes} \left( A \right) \right\}.\notag
\end{align}	
The \textbf{local spectral radius} of \( A \in M_n(\mathbb{R}_{\max}) \) at \( x \in \mathbb{R}_{\max}^n \) in max algebra is defined in \cite{MP} as:
\begin{equation*}
	r_x(A)=\underset{k\rightarrow \infty}{\limsup}\Vert {A_{\otimes} ^k} \otimes x\Vert ^{1/k};
\end{equation*}
where $  r_x(A) $ is independent of choice of any vector norm. The norm \( \|\cdot\| \) can be defined as \( \|x\| = \max_{1 \leq i \leq n} |x_i| \) for \( x \in \mathbb{C}^n \). It follows, as stated in \cite{MP}, that  
\[  
r_x(A) = \max\{r_{e_i}(A)~|~x_i \neq 0,~1 \le i \le n\}.  
\]  
As noted in \cite{MPS22}, this leads to  
\[  
r_{\otimes}\left(A\right) = \max\{r_x(A) :~ x \in \mathbb{R}_{\max}^n\}.  
\]
Finally, if $A\in M_n(\zR_{\max})$ is an irreducible matrix, then the max spectrum of $A$ has a unique element which is $r_{\otimes}\left( A \right)$.
Inspired by \cite[Remark 4.3.6]{butkovich}, in the following lemma, it is demonstrated that the spectral radius of any matrix is equal to the spectral radius of its transpose in the context of max algebra.
\begin{lem}\label{lem1.5}
	Let $ A\in M_n(\zR_{\max}) $. Then $ r_{\otimes}(A)=r_{\otimes}(A^T) $ that $ A^T $ is the transpose of $ A $.
	\begin{proof}
		The matrix $ A $ can be represented as a directed graph where the entries of $ A $ correspond to the weight of edges between vertices. The transpose of the matrix $ A  $ reverses the direction of the edges in the graph, but the weights of the edges remain the same.		
		Since the cycle and their corresponding geometric mean remain unchanged when transposing the matrix $ A $, So $ r_{\otimes}(A)=r_{\otimes}(A^T)  $.
	\end{proof}
\end{lem}
We can consider the matrix $A\in M_n(\zR_{\max})$ as a vector in $ \mathbb{R}^{n\times n} $. Accordingly; by considering the following vector norm \cite{johnson},
\[
	\| A\|=\max_{1\leq i,j\leq n}|a_{i,j}|, \qquad A\in M_n(\mathbb{C});
\]
we can prove that the norm of any max-row (column) stochastic matrix is equal to one.
\begin{thm}\label{lem1.1}
	If $ D\in  M_{n}\left( \zR _{\max} \right) $ is a max-row (column) stochastic matrix then
	\begin{align*}
		r_{\otimes}\left(D\right)=1,\quad \| D\| =1  \quad\mbox{and}\quad   d_{i,j}\leq 1\quad \mbox{for} \quad 1\leq i,j\leq n.
	\end{align*}
\end{thm}
\begin{proof}
	Let \( D \in M_n(\mathbb{R}_{\max}) \) be a max-row stochastic matrix. By Definition \ref{def2.1}, we have \( d_{i,j} \leq 1 \) for all \( i \) and \( j \), which implies \( \|D\| = 1 \). Furthermore, since \( D \otimes \mathbb{1} = \mathbb{1} \), it follows that \( 1 \in \sigma_{\otimes}(D) \), meaning \( r_{\otimes}(D) \geq 1 \). From relation \eqref{rabete1.1}, the average geometric mean of any cycle in the associated graph is at most \( 1 \), giving \( r_{\otimes}(D) \leq 1 \). Thus, \( r_{\otimes}(D) = 1 \).

	If $D$ is instead a max-column stochastic matrix, then its transpose $D^T$ is a max-row stochastic matrix. As a result, we again have $ d_{i,j}\leq 1 $, and therefore $ \| D\|=\| D^T\| =1 $.  Additionally;  since $ D^T\otimes \mathbb{1} =\mathbb{1}$, it follows that $ r_{\otimes}(D^T)=1 $. Thus, by utilizing Lemma \ref{lem1.5}, we can conclude that $ r_{\otimes}(D)=1 $.
\end{proof}
\begin{prop}\label{rem3}
	In a similar manner to the previous theorem, we can conclude that for \( D \in MDS_n(\mathbb{R}_{\max}) \):
	\begin{align*}
		r_{\otimes}\left(D\right)=1,\quad \| D\| =1  \quad\mbox{and}\quad   d_{i,j}\leq 1\quad \mbox{for} \quad 1\leq i,j\leq n.
	\end{align*}
\end{prop}
In exploring the framework of max-convex analysis, our approach is structured around foundational results that offer key insights into the subject. These results act as the cornerstones of max-convex analysis, forming a solid foundation for further exploration and investigation. Inspired by the source \cite{gabert2}, we define the following concepts.
\begin{defn}\label{dfn3}
	Let $S$ be a subset of $\zR_{max}^{n}$ or ($M_{n}\left( \zR_{\max} \right)$). The max-convex combination of finite elements  of $ S $ such $ s_1,\dots,s_i$ can be written as $ \alpha_{1} s_{1} \oplus\dots\oplus \alpha_{i} s_{i} $, where $I$ is a finite set, $ \{ s_{i}\}_{i\in I}$ is a family of elements of $S$ and $\{ \alpha_{i}\}_{i\in I}$ are scalars that satisfy $\bigoplus_{i\in I}\alpha_{i}=1$. The max-convex hull of $S$, denoted by $\mathrm{co}_{\max}\left(S\right)$, is the set of all finite max-convex combinations of elements of S. 
\end{defn}
\begin{defn}\label{dfn1}
	A set $S \subseteq\zR_{max}^{n}$ or ($M_{n}\left( \zR _{\max} \right)$) is said to be max-convex if $ \alpha_1 s_1\oplus\alpha_2  s_2 \in S $ for every $s_1,s_2\in S$ and $ \alpha_1 ,\alpha_2\in\zR_{\max} $ with $\alpha_1\oplus\alpha_2 =1 $.
\end{defn}
\begin{defn}\label{dfn2}
	 Let $ S $ be a max-convex subset of $ \zR_{\max}^{n}$ or ($M_{n}\left( \zR _{\max} \right)$). An element $ s\in S $ is called a max-extreme point of $ S $ if for all $ s_1,s_2\in S $ and $ \alpha_1 ,\alpha_2\in\zR_{\max} $ such that $ \alpha_1\oplus\alpha_2 = 1 $ and $ s=\alpha_1 s_1\oplus \alpha_2 s_2 $, it follows that $ s=s_1 $ or $ s=s_2 $. We denote the set of all max-extreme points of S with the symbol $ \mbox{ext}_{\max}\left( S \right) $.
\end{defn}
Thus, a point \( s\in S \) is called a max-extreme point if it cannot be part of a segment within \( S \) unless it is one of the endpoints of that segment. It is important to note that, due to the idempotency of addition, the property in Definition \ref{dfn2} does not imply \( s = \alpha_1 s_1 \oplus \alpha_2 s_2 \) leading to \( s = s_1 \) and \( s = s_2 \).
\begin{rem}
	Let \( S \subseteq \mathbb{R}_{\max}^{n} \) be a max-convex set, and let \( s^{(1)}, s^{(2)}, s \in S \). If \( s  \) is a max-extreme point of \( S \) and \( s = \alpha_1 s^{(1)} \oplus \alpha_2 s^{(2)} \) with \( \alpha_1 \oplus \alpha_2 = 1 \), then
	\[
	\left(s = s^{(1)}, \alpha_1 = 1\right) \quad \text{or} \quad \left(s = s^{(2)}, \alpha_2 = 1\right).
	\]
	Suppose \( s = s^{(1)} \) but \( \alpha_1 < 1 \). Consequently, \( \alpha_2 = 1 \). Assume by contradiction, \( s \neq s^{(2)} \). Then, there exists an index \( i \) with \( 1 \leq i \leq n \) such that \( s^{(2)}_i < s_i \). However, we have
	\[
	s_i = \alpha_1 s^{(1)}_{i} \oplus\alpha_2 s^{(2)}_{i} =\alpha_1 s_i\oplus s^{(2)}_i \ < s_i\oplus s_i=s_i,
	\]
	which is a contradiction and the statement holds. This observation also is valid for the set \( M_{n} \left( \mathbb{R}_{\max} \right) \).
\end{rem}
\begin{exam}
	Let $ S= \left( \left\{ 3 \right\} \times \left[1, 3 \right] \right) \cup \left\{ \left(x_1,x_2\right)^T\in\mathbb{R}^2|~x_1+2x_2\geq 7,~ x_1,x_2\leq 3  \right\}$ be a subset in $ \zR^{2}_{\max} $. We can verify that $ S $ is a max-convex set and
	$$  \mbox{ext}_{\max}\left( S \right)=\{ \left(3,1\right)^T \} \cup \left\{ \left(x_1,x_2\right)^T\in\mathbb{R}^2|~x_1+2x_2=7,~ x_1<3,~ x_2\leq 3 \right\}.$$
\begin{figure}[h]
		\centering
		\begin{minipage}{.4\textwidth}
			\centering
			\includegraphics[width=.7\linewidth]{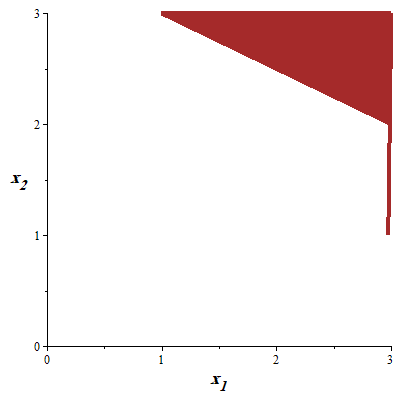}
			\captionof{figure}{{\small Set of $S$}}
			\label{fig:5}
		\end{minipage}%
		\begin{minipage}{.4\textwidth}
			\centering
			\includegraphics[width=.7\linewidth]{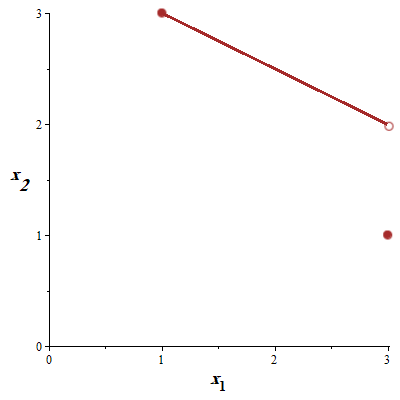}
			\captionof{figure}{{\small Set of $\mbox{ext}_{\max}\left( S \right)$}}
			\label{fig:6}
		\end{minipage}
	\end{figure}
\end{exam}
Definitions \ref{dfn1} and \ref{dfn2} provide specific criteria for identifying max-extreme points in sets within $ \zR_{\max}^{n}$ and ($M_{n}\left( \zR _{\max} \right)$). These definitions serve as valuable tools for characterizing and distinguishing max-extreme points in the context of max-convex analysis.
\begin{thm}\label{lem1.2}
	The set of max-doubly stochastic matrices is max-convex.
\end{thm}
\begin{proof}		
	To prove the max-convexity of the set of max-doubly stochastic matrices, it is sufficient to verify that the condition stated in Definition \ref{dfn1} is satisfied. Consider two matrices \( D_1,D_2\in MDS_n(\mathbb{R}_{\max}) \) and let \( \alpha_1, \alpha_2 \in \mathbb{R}_{\max} \) such that \( \alpha_1 \oplus \alpha_2 = 1 \). Since both \( D_1 \) and \( D_2 \) are max-row stochastic matrices, we have the following:
	\begin{align*}
	\left( \alpha_1 D_1\oplus \alpha_2 D_2  \right)\otimes\mathbb{1} &= \alpha_1 \left( D_1\otimes\mathbb{1}\right) \oplus \alpha_2 \left( D_2\otimes\mathbb{1} \right)\\
	&=\alpha_1\mathbb{1}\oplus\alpha_2\mathbb{1}\\
	&= \left(\alpha_1\oplus\alpha_2\right)\mathbb{1}\\
	&=\mathbb{1}.
\end{align*}
	Thus, \( \alpha_1 D_1 \oplus \alpha_2 D_2 \) satisfies the condition for being a max-row stochastic matrix. Similarly, the same reasoning can be applied to show that $ \alpha_1 D_1\oplus \alpha_2 D_2 $ is a max-column stochastic matrix. Therefore $ \alpha_1 D_1 \oplus \alpha_2 D_2\in MDS_n(\mathbb{R}_{\max})  $ and this means that the set of max-doubly stochastic matrices is max-convex.
\end{proof}
According to \cite{butkovich}, the identity matrix and permutation matrix in \( M_n(\mathbb{R}_{\max}) \) are defined as follows:
\begin{defn}
	For \( I_n, P \in M_n(\mathbb{R}_{\max}) \):
	\begin{enumerate}
		\item The matrix \( I_n \) is called the max-identity matrix (or simply the identity matrix) if its diagonal entries are \( 1 \) and all off-diagonal entries are \( 0 \).
		\item The matrix \( P \) is called a max-permutation matrix (or permutation matrix) if it is obtained from the identity matrix by permuting its rows and/or columns.
	\end{enumerate}
\end{defn}
It is evident that every permutation matrix is a max-doubly stochastic matrix.
Moreover, multiplying a permutation matrix on the left side of an arbitrary matrix in \( M_n(\mathbb{R}_{\max}) \) permutes its rows, while multiplying it on the right side permutes its columns. From this observation, we conclude that the product of a finite number of permutation matrices is itself a permutation matrix. We denote the set of all permutation matrices of size \( n \) by \( \mathbf{P}_{n}(\mathbb{R}_{\max}) \).
Finally, a \( \mathbf{(0,1)} \)\textbf{-matrix} is defined as a matrix whose entries are exclusively 0 and 1.

In the following theorem and lemma, we identify the max-extreme points of the set \( MDS_n(\mathbb{R}_{\max}) \).
\begin{lem}\label{lem5}
	Every max-extreme point of $ MDS_n(\mathbb{R}_{\max}) $ is a $(0,1)$-matrix.
	\begin{proof}
		We will prove this by contradiction. Assume that there exists a max-extreme point of the set \( MDS_n(\mathbb{R}_{\max}) \) such $ E=\left( e_{i,j} \right)_{n\times n}  $ with an entry \( e_{i_0,j_0} = \alpha \), where \( 0 < \alpha < 1 \). We can express \( E \) as	
		\[  
		E = D_1 \oplus \alpha D_2,  
		\]
		where \( D_1 \) is derived from \( E \) by replacing \( e_{i_0,j_0} \) with \( 0 \) (i.e., \( d^{(1)}_{i_0,j_0} = 0 \)) and  \( D_2 \) is derived from $ E $ by replacing \( e_{i_0,j_0} \) with $ 1 $ (i.e., \( d^{(2)}_{i_0,j_0} = 1 \)).
		It is clear that both matrices \( D_1 \) and \( D_2 \) are max-doubly stochastic matrices and \( E \) is a max-convex combination of $ D_1 $ and $ D_2 $. This contradicts the assumption that $ E $ is a max-extreme point of the set \( MDS_n(\mathbb{R}_{\max}) \).  Thus, the max-extreme points of the set of max-doubly stochastic matrices contain only zero and one entries.
	\end{proof}
\end{lem}
\begin{thm}\label{theorem1}
	The matrix \( E =\left( e_{i,j} \right)_{n\times n}\in MDS_n(\mathbb{R}_{\max}) \) is a max-extreme point of the set \( MDS_n(\mathbb{R}_{\max}) \) if and only if \( E \) is a \( (0,1) \)-matrix and there is at most one entry $e_{i,j}=1$ such that $E$ remains max-doubly stochastic matrix after changing $e_{i,j}$ to a number $0\leq r<1$.
	\begin{proof}
		Suppose \( E = \left( e_{i,j} \right)_{n \times n} \in MDS_n(\mathbb{R}_{\max}) \) is a max-extreme point of the set of max-doubly stochastic matrices. By Lemma \ref{lem5}, \( E \) is a \( (0,1) \)-matrix. 
		To prove the second condition, assume for contradiction that \( E \) has at least two entries equal to one where if these entries are changed to a number $0\leq r<1$, then $ E $ is still a max-doubly stochastic matrix.
		Without losing any generality, tha matrix \( E \) has entries  \( e_{i_1,j_1}\) and \( e_{i_2,j_2} \) equal to one such that $E$ remains max-doubly stochastic matrix after changing these intries to a number $0\leq r<1$.
		 Let
		\( D_1=\left(d_{i,j}^{(1)}\right)_{n\times n}, D_2=\left(d_{i,j}^{(2)}\right)_{n\times n} \in MDS_n(\mathbb{R}_{\max}) \) be matrices obtained from \( E \) that $ D_1 $ is the matrix $ E $ with replacing \( e_{i_1,j_1} \) with \( r_1 \) (i.e., \( 0<d^{(1)}_{i_1,j_1} = r_1 <1\)) and the matrix $ D_2 $ is obtained in a similar way, too (i.e., \( 0<d^{(2)}_{i_2,j_2} = r_2 <1\)). \( E = D_1 \oplus D_2 \) because \( e_{i_1,j_1} = \max(1, r_1) = 1 \) and \( e_{i_2,j_2} = \max(1, r_2) = 1 \). Consequently, \( E \) is a max-convex combination of \( D_1 \) and \( D_2 \). This contradicts the assumption that \( E \) is a max-extreme point of \( MDS_n(\mathbb{R}_{\max}) \).
		Thus, \( E \) consists only of entries \( 0 \) and \( 1 \), and changing at most one entry equal to \( 1 \) to \( 0 \leq r < 1 \) still results in a max-doubly stochastic matrix.
		
		Conversely, suppose \( E \in MDS_n(\mathbb{R}_{\max}) \) is a \( (0,1) \)-matrix and changing at most one entry equal to \( 1 \) to \( 0 \leq r < 1 \) still results in a max-doubly stochastic matrix, we will prove that $ E $ is a max-extreme point of the set \( MDS_n(\mathbb{R}_{\max}) \).
		Assume \( E \) can be expressed as a max-convex combination of matrices in \( MDS_n(\mathbb{R}_{\max}) \), i.e., \( E = \alpha_1 D_1 \oplus \alpha_2 D_2 \), where \( \alpha_1 \) and \( \alpha_2 \) are non-negative scalars with \( \alpha_1 \oplus \alpha_2 = 1 \), and \( D_1=\left(d_{i,j}^{(1)}\right)_{n\times n}, D_2=\left(d_{i,j}^{(2)}\right)_{n\times n} \in MDS_n(\mathbb{R}_{\max}) \).
		If $ \alpha_1=0 $ or $ \alpha_2=0 $ then the conclusion is trivial. So we can assume $ \alpha_1,\alpha_2>0 $.
		Based on the following reasons, the only possibility is \( ( \alpha_1 = 1,E = D_1) \) or \( ( \alpha_2 = 1,E = D_2) \):
		\begin{itemize}
			\item
			\( D_1 \) and \( D_2 \) must have zeros in the zero positions of \( E \), since \( E \geq \alpha_1 D_1, E\geq \alpha_2 D_2\).
			\item 
			Consider the entries of \( E \) that are equal to \( 1 \). If \( 0 < \alpha_1 < 1 \), then \( \alpha_2 = 1 \). Consequently, \( D_2 \) must have the same positions of entries equal to \( 1 \) as \( E \), since by Proposition \ref{rem3}, all entries of the matrix \( D_1 \) and $ D_2 $ lie between \( 0 \) and \( 1 \), and the matrix \( E \) is the max-convex combination of \( D_1 \) and \( D_2 \). By considering the previous part, this implies that \( E = D_2 \). Similarly, for \( 0 < \alpha_2 < 1 \), it can be shown that \( E = D_1 \).
			Now, suppose \( \alpha_1 = \alpha_2 = 1 \). We claim that \( E = D_1 \) or \( E = D_2 \). Assume, for contradiction, that \( E \neq D_1 \) and \( E \neq D_2 \). Since the entries of \( D_1 \) and \( D_2 \) are zero at the zero positions of \( E \), the condition \( E \neq D_1 \) implies the existence of some \( 1 \leq i, j \leq n \) such that \( e_{i,j} = 1 \) but \( d_{i,j}^{(1)} = r \) with \( 0 \leq r < 1 \). Similarly, \( E \neq D_2 \) implies the existence of some positions \( (i, j) \) where \( e_{i,j} = 1 \) and \( d^{(2)}_{i,j} = r \) with \( 0 \leq r< 1 \).
			Without loss of generality, suppose there exist indices $ (i_1,j_1) $ and $ (i_2,j_2) $ such that $ e_{i_1,j_1}=e_{i_2,j_2} = 1 $ and $ d^{(1)}_{i_1,j_1} = r_1 $, $ d^{(2)}_{i_2,j_2} = r_2 $ with $ 0\leq r_1,r_2<1 $.
			Define the matrix \( E_1 \) by replacing the \( (i_1, j_1) \)-th entry of \( E \) with \( r_1 \), i.e., \( e^{(1)}_{i_1,j_1} = r_1 \). Then, we have \( D_1 \leq E_1 \leq E \), and by Theorem \ref{theorem6}, \( E_1 \in MDS_n(\mathbb{R}_{\max}) \). Similarly, define \( E_2 \) by replacing \( e_{i_2,j_2} \) in \( E \) with \( r_2 \), i.e., \( e^{(2)}_{i_2,j_2} = r_2 \). Again, we obtain \( D_2 \leq E_2 \leq E \) and \( E_2 \in MDS_n(\mathbb{R}_{\max}) \).  
			This gives there are two entries equal to $ 1 $ in \( E \) that can be changed into values $ 0\leq r<1 $ and the result would still be a max-doubly stochastic matrix, which contradicts over assumption for $ E $.			
		\end{itemize}
	\end{proof}
\end{thm}
Suppose \( A = \left( a_{i,j} \right)_{m \times n} \) is a \( (0,1) \)-matrix in max algebra. An entry \( a_{i,j} \) is called a \textbf{row singleton entry} if it is the only non-zero element in the \( i \)-th row of the matrix \( A \). Similarly, \( a_{i,j} \) is a \textbf{column singleton entry} if it is the only non-zero element in the \( j \)-th column of the matrix \( A \). An entry \( a_{i,j} \) that is either a row singleton entry or a column singleton entry is referred to as a \textbf{singleton entry} of matrix \( A \).
An entry equal to 1 that is not a singleton is referred to as a \textbf{non-singleton entry}.

The \textbf{direct sum} of two matrices \( A \) and \( B \) in max algebra, denoted by \( A \boxplus B \), is defined as the block matrix:
\[
A \boxplus B = \begin{pmatrix}
	A & 0 \\
	0 & B
\end{pmatrix},
\]
where \( A \) is an \( m \times n \) matrix and \( B \) is a \( t \times k \) matrix in max algebra. The zero blocks have dimensions \( m \times k \) and \( t \times n \), respectively. The vector \( \mathbb{1}_{m} \) is the column vector of size \( m \) with all elements equal to one.

Based on these definitions, Theorem \ref{theorem1} can be restated as follows: The max-extreme points of the set $ MDS_n(\mathbb{R}_{\max}) $, are $(0,1)$-matrices which every non-zero entry is a singleton entry, or $(0,1)$-matrices that have exactly one non-singleton entry equal to one. The following lemma gives the explicit form of these two types of $(0,1)$-matrices, even in the more general case of non-square matrices.

\begin{lem}\label{theorem4}
	Consider the following two types of  matrices in max algebra:
	\begin{itemize}
		\item \textbf{(Type 1):} The \( (0,1) \)-matrix $A$ of size $m\times n$ in max algebra, satisfying $\mathbb{1}_{m}^T\otimes A=\mathbb{1}_{n}^T,~A\otimes\mathbb{1}_{n}=\mathbb{1}_{m}$, where every non-zero entry in $A$ is singleton.
		\item \textbf{(Type 2):} The \( (0,1) \)-matrix $B$ of size $m\times n$ in max algebra, satisfying $\mathbb{1}_{m}^T\otimes B=\mathbb{1}_{n}^T,~B\otimes\mathbb{1}_{n}=\mathbb{1}_{m}$, where $B$ has exactly one non-singleton entry.
	\end{itemize}
	Then, the matrix $ A $ can be expressed as  $ P_1\otimes E_1\otimes P_2 $ and the matrix $ B $ as $ P_1\otimes E_2\otimes P_2 $, where $ P_1\in\mathbf{P}_{m}(\mathbb{R}_{\max}) $, $ P_2\in\mathbf{P}_{n}(\mathbb{R}_{\max}) $ and $ E_1 $ and $ E_2 $ have the following structure:
	\begin{equation}\label{rabete4}
		E_1= \mathbb{1}_{ m_{1}}\boxplus\cdots\boxplus \mathbb{1}_{ m_{k}}\boxplus \mathbb{1}^T_{ n_{1}}\boxplus \cdots \boxplus \mathbb{1}^T_{ n_{t}},
	\end{equation}
	for $ t,k\geq 0 $,  and
	\begin{equation}\label{rabete5}
		E_2=\begin{pmatrix} 
			1 & 1 & \cdots & 1  \\
			1 & 0 & \cdots & 0 \\
			\vdots & \vdots & \ddots & \vdots \\
			1 & 0 & \cdots & 0
		\end{pmatrix}_{q\times r} \boxplus \mathbb{1}_{ m_{1}}\boxplus\cdots\boxplus \mathbb{1}_{ m_{k}}\boxplus \mathbb{1}^T_{ n_{1}}\boxplus \cdots \boxplus \mathbb{1}^T_{ n_{t}},
	\end{equation}
	for $t,k\geq 0$ and $ 1<q,r$.
\begin{proof}
	Suppose \( A = \left( a_{i,j} \right)_{m \times n} \) is an $m\times n$ matrix of (Type 1). By applying suitable permutations to the rows and columns of \( A \), assume $ a_{1,1}=1 $. One of the following three situations occurs:
	\begin{enumerate}
		\item $ a_{1,1} $ is a row singelton:
		By applying suitable permutations to the rows of \( A \), assume:
		\[
		a_{1,1} = a_{2,1} = \cdots = a_{m_1,1} = 1,
		\]
		\[
		a_{m_1 + 1, 1} = a_{m_1 + 2, 1} = \cdots = a_{m, 1} = 0.
		\]
		Since \( A \) contains no non-singleton entries, all \( a_{i,1} \) (\( 1 \leq i \leq m_1 \)) must be row singletons. Consequently, the remaining entries in the first \( m_1 \) rows of \( A \) are zero. Thus, \( A \) can be expressed as $ P_1\otimes\left(\mathbb{1}_{m_1}\boxplus A^{(1)}\right)\otimes P_2,$
		where \( A^{(1)}=\left(a^{(1)}_{i,j}\right)_{(m-m_1)\times (n-1)} \) is a \( (0,1) \)-matrix in (Type 1) and \( P_1 \in \mathbf{P}_m(\mathbb{R}_{\max}) \), \( P_2 \in \mathbf{P}_n(\mathbb{R}_{\max}) \).
		\item $ a_{1,1} $ is a column singelton:
		 after applying appropriate permutations to the columns of \( A \) besides the first column, we can assume:
		\[
		a_{1,1} = a_{1,2} = \cdots = a_{1,n_1} = 1,
		\]
		\[
		a_{1,n_1 + 1} = a_{1,n_1 + 2} = \cdots = a_{1, n} = 0.
		\]
		All \( a_{1,j} \) (\( 1 \leq j \leq n_1 \)) must be column singletons because \( A \) does not have any non-singleton entries. consequently all of the remaining entries in the first $n_1$ column are zero. Thus $A$ can be expressed as $ P_1\otimes\left(\mathbb{1}^T_{n_1}\boxplus A^{(1)}\right)\otimes P_2 $, where \( A^{(1)}=\left(a^{(1)}_{i,j}\right)_{(m-1)\times (n-n_{1})}\) is a \((0,1)\)-matrix belonging to the first type, and $ P_1\in\mathbf{P}_{m}(\mathbb{R}_{\max}) $ and $ P_2\in\mathbf{P}_{n}(\mathbb{R}_{\max}) $.
		\item  $ a_{1,1} $ is both a row and a column singelton:
		it is clear that $A$ can be expressed as $ P_1\otimes\left(\mathbb{1}_1\boxplus A^{(1)}\right)\otimes P_2 $, where \( A^{(1)}=\left(a^{(1)}_{i,j}\right)_{(m-1)\times (n-1)} \) is a \((0,1)\)-matrix belonging to (Type 1), and $ P_1\in\mathbf{P}_{m}(\mathbb{R}_{\max}) $ and $ P_2\in\mathbf{P}_{n}(\mathbb{R}_{\max}) $.
	\end{enumerate}
	By using induction, after applying permutations to the rows and columns of \( A^{(1)} \), the entry $ a^{(1)}_{1,1} $ becomes $ 1 $ and one of the three above states occurs again.
	By continuing this process and using appropriate permutation matrices, finally, we can consider the matrix $ A $ as $ P_1\otimes E_1\otimes P_2 $ where, \( P_1 \in \mathbf{P}_m(\mathbb{R}_{\max}) \) and \( P_2 \in \mathbf{P}_n(\mathbb{R}_{\max}) \) and $ E_1 $ has the structure described in relation \ref{rabete4}.
	
	Now, Let $ B= \left( b_{i,j} \right)_{m \times n} $ be a matrix belonging to (Type 2). By applying the appropriate permutation to the rows and columns of the matrix $ B $, we can assume that $ b_{1,1}=1 $ is the only non-singleton entry. By using appropriate row and column permutations, we can assume that:
	\[
	b_{1,1}=b_{1,2}=\cdots=b_{1,r}=1,~~ b_{1,r+1}=\cdots=b_{1,n}=0,
	\]
	\[
	b_{1,1}=b_{2,1}=\cdots=b_{q,1}=1,~~ b_{q+1,1}=\cdots=b_{m,1}=0,
	\]
	where $ 1<q\leq m $ and $ 1< r\leq n $. Since we have more than one entries equal to $ 1 $ in the first row, all of these  are column singleton entries, and so in the columns $ 2 $ to $ r $, the remaining entries of the matrix $ B $ are equal to zero. Similarly, all entries $ b_{i,1} $ ($ 2\leq i\leq q $) are row singleton, and so all other entries in the rows $ 2 $ to $ q $, are equal to zero. Hence there are matrices $ B^{(1)} $ and $ B^{(2)} $, such that $ B^{(2)} $ is a matrix of size $ q\times r $ and $ B^{(1)} $ is a matrix of size $ (m-q)\times (n-r) $ in max algebra, and  permutation matrices \( P_1,P_2\) of suitable size, such that
		\begin{equation*}
		B^{(2)}=\begin{pmatrix} 
			1 & 1 & \cdots & 1  \\
			1 & 0 & \cdots & 0 \\
			\vdots & \vdots & \ddots & \vdots \\
			1 & 0 & \cdots & 0
		\end{pmatrix}_{q\times r},
	\end{equation*}
	where $ 1<q\leq m $, $ 1< r\leq n $ and
	\[
	B=P_1\otimes \left(B^{(2)}\boxplus B^{(1)}\right)\otimes P_2.
	\]
	It is easily to see that $ B^{(1)} $ is a matrix of the first type. By considering the previous part and applying appropriate row and column permutations, the matrix $ B^{(1)} $ can be written in the form of the relation \ref{rabete4}.  Finally, we can represent the matrix $ B $ as $P_1\otimes E_2 \otimes P_2$,	where \( P_1 \in \mathbf{P}_{m}(\mathbb{R}_{\max}) \), \( P_2 \in \mathbf{P}_{n}(\mathbb{R}_{\max}) \) and the matrix $ E_2 $ is of form \ref{rabete5}.
\end{proof}
\end{lem}
Using Theorem \ref{theorem1} and Lemma \ref{theorem4}, we can derive the general form of max-extreme points of $MDS_n(\mathbb{R}_{\max})$.
\begin{thm}\label{extremthm}
	All of the max-extreme points of $ MDS_n(\mathbb{R}_{\max}) $, are of the form $ P_1\otimes E_1\otimes P_2 $ or $ P_1\otimes E_2\otimes P_2 $, such that $ P_1,P_2\in\mathbf{P}_{n}(\mathbb{R}_{\max}) $ and $ E_1 $ and $ E_2 $ are of the following forms:
	\begin{equation*}
		E_1= \mathbb{1}_{ m_{1}}\boxplus\cdots\boxplus \mathbb{1}_{ m_{k}}\boxplus \mathbb{1}^T_{ n_{1}}\boxplus \cdots \boxplus \mathbb{1}^T_{ n_{t}},
	\end{equation*}
	where $ t,k\geq 0 $, and
	\begin{equation*}
		E_2=\begin{pmatrix}
			1 & 1 & \cdots & 1  \\
			1 & 0 & \cdots & 0 \\
			\vdots & \vdots & \ddots & \vdots \\
			1 & 0 & \cdots & 0
		\end{pmatrix}_{q\times r} \boxplus \left( \mathbb{1}_{ m_{1}}\boxplus\cdots\boxplus \mathbb{1}_{ m_{k}}\boxplus \mathbb{1}^T_{ n_{1}}\boxplus \cdots \boxplus \mathbb{1}^T_{ n_{t}}\right),
	\end{equation*}
	where $ 1<q, r $ and $t,k\geq 0$.
\end{thm}
\begin{exam}
	According to Theorem \ref{extremthm}, the max-extreme points of the set $  MDS_3(\mathbb{R}_{\max}) $ and $  MDS_2(\mathbb{R}_{\max}) $ are as follows. Let
	{\fontsize{7.5}{.2}
	\[
	\mathcal{P}_1=
	\left\{\begin{pmatrix}
		1 & 0 & 0 \\
		0 & 1 & 0 \\
		0 & 0 & 1	
	\end{pmatrix},~~
	\begin{pmatrix}
		1 & 1 & 0 \\
		1 & 0 & 0 \\
		0 & 0 & 1	
	\end{pmatrix},~~
	\begin{pmatrix}
		0 & 1 & 1 \\
		1 & 0 & 0 \\
		1 & 0 & 0
	\end{pmatrix},~~
	\begin{pmatrix}
		1 & 1 & 1 \\
		1 & 0 & 0 \\
		1 & 0 & 0	
	\end{pmatrix}\right\},
	\]
	}
	and
	\[
		\mathcal{P}_2=
	\left\{\begin{pmatrix}
		1 & 0 \\
		0 & 1 
	\end{pmatrix},~~
	\begin{pmatrix}
		1 & 1  \\
		1 & 0
	\end{pmatrix}\right\},
	\]
	then
	\[
	\mbox{ext}_{\max}\left( MDS_3(\mathbb{R}_{\max}) \right)= \left\{ P_1\otimes E\otimes P_2 ~|~E\in\mathcal{P}_1,~ P_1,P_2\in \mathbf{P}_{3}(\mathbb{R}_{\max})\right\},
	\]
	and
		\[
	\mbox{ext}_{\max}\left( MDS_2(\mathbb{R}_{\max}) \right)= \left\{ P_1\otimes E\otimes P_2 ~|~E\in\mathcal{P}_2,~ P_1,P_2\in \mathbf{P}_{2}(\mathbb{R}_{\max})\right\}.
	\]		
\end{exam}
In the context of doubly stochastic matrices, Birkhoff's theorem (also known as the Birkhoff-von Neumann theorem) states that any doubly stochastic matrix can be written as a convex combination of permutation matrices \cite{batia}. Based on the example above, we observe that even in the case of \( n=2 \), there exist max-extreme points in \( MDS_n(R_{\max}) \) that are not permutation matrices. Thus, Birkhoff's theorem does not hold in this version.

\section{\textbf{Majorization in Max Algebra}}
By utilizing max-doubly stochastic matrices, we establish a framework for majorization within the semiring \( \mathbb{R}_{\max}^n \). Furthermore, we characterize the elements of the set of
$\left\{ x \in \mathbb{R}_{\max}^n \mid x \prec_{\max} y \right\}$
for an arbitrary \( y \in \mathbb{R}_{\max}^n \). To enhance understanding of max-majorization, we provide examples along with key theorems related to this topic.
\begin{defn}\label{def3.3}
Let $ x,y\in\zR _{\max}^{n} $. We say that $x$ is max-majorized by $ y $ in symbol $ x\prec_{\max} y $, if there is a max-doubly stochastic matrix such $ D $ that $ x=D\otimes y $.
\end{defn}
The following theorem demonstrates how vectors are classified according to the definition of max-majorization.
\begin{thm}\label{theorem2}
	For $ x,y\in\zR _{\max}^{n} $, $x$ is max-majorized by $ y $ if and only if
	\[
	 \max_{1\leq i\leq n}x_i=\max_{1\leq i\leq n}y_i  \qquad\mbox{and}\qquad  \min_{1\leq i\leq n}x_i\geq\min_{1\leq i\leq n}y_i.
	\]
\begin{proof}
	Let \( x = \left(x_1, x_2, \dots, x_n\right)^T, y= \left(y_1, y_2, \dots, y_n\right)^T \in \mathbb{R}_{\max}^n \) with \( x \prec_{\max} y \). By Definition \ref{def3.3}, there exists a max-doubly stochastic matrix \( D=\left(d_{i,j}\right)_{1\leq i,j\leq n} \) such that \( x = D \otimes y \). For \( 1 \leq i \leq n \), there exists \( 1 \leq t_i \leq n \) that
	\[
	x_i = \max_{1 \leq j \leq n} \left( d_{i,j} y_j \right)=d_{i,t_i} y_{t_i}\leq y_{t_i}\leq \max_{1 \leq i \leq n} y_i,
	\]
	where the inequality follows from \( d_{i,j} \leq 1 \) by Proposition \ref{rem3}. Thus,
	\[ \max_{1 \leq i \leq n} x_i \leq \max_{1 \leq i \leq n} y_i.\]
	Now, let \( y_h = \max_{1 \leq i \leq n} y_i \). Since \( D \) is a max-doubly stochastic matrix, there must be a coordinate equal to 1 in the \( h \)-th column of \( D \). That is, there exists an index \( 1 \leq k \leq n \) such that \( d_{k,h} = 1 \). Therefore,
	\[
	\max_{1 \leq i \leq n} x_i \geq x_k = \max_{1 \leq j \leq n} \left( d_{k,j} y_j \right) \geq d_{k,h} y_h = y_h = \max_{1 \leq i \leq n} y_i.
	\]
	Combining these results, we conclude that \( \max_{1 \leq i \leq n} x_i = \max_{1 \leq i \leq n} y_i \).
	Now, we aim to prove that \( \min_{1 \leq i \leq n} x_i \geq \min_{1 \leq i \leq n} y_i \). For all \( 1 \leq i \leq n \), we have
	$x_i = \max_{1 \leq j \leq n} \left( d_{i,j} y_j \right).$
	As \( D \) is a max-doubly stochastic matrix, there exists an index \( 1 \leq t_i \leq n \) such that \( d_{i,t_i} = 1 \). Therefore,
	\[
	x_i = \max_{1 \leq j \leq n} \left( d_{i,j} y_j \right) \geq d_{i,t_i} y_{t_i} = y_{t_i} \geq \min_{1 \leq i \leq n} y_i.
	\]
	Since the inequality above holds for all \( 1 \leq i \leq n \), we can conclude that \( \min_{1 \leq i \leq n} x_i \geq \min_{1 \leq i \leq n} y_i \).
		
	Conversely, for \( x, y \in \mathbb{R}^n_{\max} \), suppose that
	\[
	\max_{1 \leq i \leq n} x_i = \max_{1 \leq i \leq n} y_i \quad \text{and} \quad \min_{1 \leq i \leq n} x_i \geq \min_{1 \leq i \leq n} y_i.
	\]
	We aim to prove that \( x \prec_{\max} y \). In other words, we must show that there exists a max-doubly stochastic matrix \( D \) such that \( x = D \otimes y \).
	Let \( \max_{1 \leq i \leq n} x_i = x_k \), \( \max_{1 \leq i \leq n} y_i = y_l \) and \( \min_{1 \leq i \leq n} y_i = y_m \). Define the matrix \( D \) as follows:
	\begin{itemize}
		\item For \( 1 \leq j \leq n \), let \( d_{k,j} = 1 \).
		\item For \( 1 \leq i \leq n \), let \( d_{i,m} = 1 \).
		\item For \( 1 \leq i \leq n \), let \( d_{i,l} = \frac{x_i}{y_l} \).
		\item Set all other entries of \( D \) to zero.
	\end{itemize}
	Note that for all \( 1 \leq i \leq n \), $x_i\leq x_k=y_l$, so $\frac{x_i}{y_l}\leq 1$. Also, $ d_{k,l}=\frac{x_k}{y_l}=1 $ and all of entries in the $k$-th row and $m$-th column of $D$ are equal to $1$. Hence
	\[
	\biggl(D\otimes\mathbb{1}\biggr)_{1\leq i\leq n}=\begin{cases}
		\max\left(1,0,\frac{x_i}{y_l}\right)=1, & 1\leq i \leq n,~ i\neq k, \\
		\max\left(1,1,\cdots,1\right)=1, & i=k.
	\end{cases}
	\]
	This means that $ D $ is a max-row stochastic matrix. On the other hand, 
	\[
	\biggl(\mathbb{1}^T \otimes D\biggr)_{1\leq i\leq n}=\begin{cases}
		\max\left(1,0\right)=1, & 1\leq i \leq n,~ i\neq m,l, \\
		\max\left(\frac{x_1}{y_l},\frac{x_2}{y_l},\cdots,\frac{x_n}{y_l}\right)=1, & i=l, \\
		\max\left(1,1,\cdots,1\right)=1, & i=m.
	\end{cases}
	\]
	Also $ D $ is a max-column stochastic matrix and hence $D\in MDS_n(\mathbb{R}_{\max})$.
	Considering for $ 1\leq i \leq n $,
	\[
	x_i\geq \min_{1 \leq i \leq n} x_i \geq \min_{1 \leq i \leq n} y_i=y_m
	\]
	the following result is obtained:
	\[
	\biggl(D\otimes y\biggr)_{1\leq i\leq n}=\biggl(\bigoplus_{j=1}^{n}d_{i,j}y_j\biggr)_{1\leq i\leq n}=\begin{cases}
		\max(y_m, \dfrac{x_i}{y_l}\times y_l) = x_i,&  1 \leq i \leq n,~ i\neq k, \\
		\max_{1 \leq i \leq n}y_i=\max_{1 \leq i \leq n} x_i=x_k, & i=k.
	\end{cases}
	\]
	Therefore, \( x = D \otimes y \), which implies \( x \prec_{\max} y \).
\end{proof}
\end{thm}
The relation $\prec$ is reflexive and transitive but not a partial order on $\mathbb{R}^n$. Specifically, if \(x \prec y\) and \(y \prec x\), we can only conclude that \(x = Py\) for some permutation matrix \(P\). Defining the equivalence relation \(x \sim y\) by \(x = Py\) for some permutation matrix \(P\), we obtain the quotient space \(\mathbb{R}^n_{\text{sym}}\). On this space, the relation \(\prec\) induces a partial order. Furthermore, \(\prec\) forms a partial order on the subset \(\{x \in \mathbb{R}^n : x_1 \geq x_2 \geq \cdots \geq x_n\}\) \cite{batia,marshal}.  
In max algebra, the relation \(\prec_{\max}\) is reflexive and transitive. However, Theorem \ref{theorem2} establishes that if \(x \prec_{\max} y\) and \(y \prec_{\max} x\), it does not necessarily follow that the vector \( x \) is a permutation of the vector \( y \).  

In classical linear algebra, a vector \(x \in \mathbb{R}^n\) is majorized by \(y \in \mathbb{R}^n\) if and only if \(x\) belongs to the convex hull of all vectors obtained by permuting the coordinates of \(y\) \cite{batia,marshal}. However, as demonstrated by the following theorem, this characterization does not extend to  the context of max algebra.

Suppose \( y=\left(y_1,\cdots,y_n\right)^T \in \mathbb{R}^n_{\max} \). Define
\begin{equation}
	y_{\mathrm{min}} = \min_{1 \leq i \leq n} y_i, \qquad y_{\mathrm{max}} = \max_{1 \leq i \leq n} y_i.
\end{equation}
For $ 1\leq i\leq n $, let \( y^{(i)} \in \mathbb{R}^n_{\max} \) be the vector defined as
\begin{equation}\label{rabete6}
	y^{(i)} = \left( y_{\mathrm{min}}, \dots, y_{\mathrm{min}}, y_{\mathrm{max}}, y_{\mathrm{min}}, \dots, y_{\mathrm{min}} \right)^T,
\end{equation}
where the \( i \)-th entry of \( y^{(i)} \) is \( y_{\mathrm{max}} \), and all other entries equal to \( y_{\mathrm{min}} \).
We require following lemma to prove the following theorem, which relates the set \( \left\{ x \in \mathbb{R}_{\max}^{n} \mid x \prec_{\max} y \right\} \) to the vectors \( \left\{ y^{(i)} \right\}_{i=1}^{n} \).
\begin{lem}\label{lem6}
	Let $ x_1,x_2,y\in\mathbb{R} _{\max}^{n} $ such that $ x_1,x_2\prec_{\max} y $ and let $ \alpha_1,\alpha_2\in\mathbb{R}_{\max} $ satisfy $ \alpha_1\oplus\alpha_2=1 $. Then
	\[
	\alpha_1 x_1\oplus\alpha_2 x_2\prec_{\max} y.
	\]
	\begin{proof}
		Since $ x_1,x_2\prec_{\max} y $, by Definition \ref{def3.3}, there exist max-doubly stochastic matrices $ D_1,D_2 $ such that $ x_1=D_1\otimes y $ and $ x_2=D_2\otimes y $. Using this, we obtain
		\[
		\alpha_1 x_1\oplus\alpha_2 x_2=\left( \alpha_1 D_1\oplus\alpha_2 D_2 \right)\otimes y.
		\]
		Since $ MDS_n\left(\mathbb{R}_{\max}\right) $ is a max-convex set (by Theorem \ref{lem1.2}), it follows that $ \alpha_1 D_1\oplus\alpha_2 D_2 \in MDS_n(\mathbb{R}_{\max}) $, which implies $ \alpha_1 x_1\oplus\alpha_2 x_2 \prec_{\max} y $.
	\end{proof}
\end{lem}
\begin{thm}\label{theorem5}
 Let $ y \in \mathbb{R}^n_{\max} $. Then
		\begin{align*}
			\left\{ x\in\mathbb{R} _{\max}^{n} \mid x\prec_{\max} y \right\} =
			\mathrm{co}_{\max}\left(\{y^{(i)}\}_{i=1}^n\right).
		\end{align*}
\begin{proof}
 Let $ y=\left(y_1,y_2,\dots,y_n\right)^T\in \mathbb{R}^n_{\max} $, and define
 \begin{equation*}
	\Lambda_1=\left\{ x\in\mathbb{R}_{\max}^{n} \mid x\prec_{\max} y \right\},\quad  \Lambda_2=\mathrm{co}_{\max}\left(\{y^{(i)}\}_{i=1}^n\right).
 \end{equation*}
 We aim to prove that $ \Lambda_1\subseteq\Lambda_2 $ and $ \Lambda_2\subseteq\Lambda_1 $. Suppose $ x\in\Lambda_2 $. By definition \ref{dfn3}, we have
 $$ x=\bigoplus_{i=1}^m\alpha_{i}y^{(i)}, $$
 where $ 1\leq m \leq n $, $ \bigoplus_{i=1}^{m}\alpha_{i}=1 $, $ \alpha_i\in\zR_{\max} $ and $ y^{(i)} $ is obtained by relatioin \ref{rabete6}. Suppose $ y_{\mathrm{min}} $ is $ k $-th entry of the vector $ y $ and we know $ y_i^{(i)}=y_{\mathrm{max}} $. Hence we define the matrix $ D^{(i)}\in M_n\left(\zR_{\max}\right) $ as follows:
 \begin{itemize}
	\item all entries in the $ k $-th column of $ D^{(i)} $ are equal to $ 1 $.
	\item all entries in the $ i $-th row of $ D^{(i)} $ are equal to $ 1 $.
	\item all other entries of $ D^{(i)} $ are equal to $ 0 $.
 \end{itemize}
 It is evident that $ D^{(i)}\in MDS_n(\mathbb{R}_{\max}) $ and $ y^{(i)}=D^{(i)}\otimes y $.
 Since each $ y^{(i)}\prec_{\max}y $ and $ \bigoplus_{i=1}^{m}\alpha_{i}=1 $, Lemma \ref{lem6} implies that
$$ \bigoplus_{i=1}^m\alpha_{i}y^{(i)}=x\prec_{\max}y. $$‌
 Therefore $ x\in\Lambda_1 $, and we conclude that $ \Lambda_2\subseteq\Lambda_1 $.
	
 To prove the reverse inclusion, let $x=\left(x_1,\dots,x_n\right)^T\in\Lambda_1$. Then by Theorem \ref{theorem2}, we have $x_{\max}=y_{\max}$ and $ x_{\min}\geq y_{\min}$. For every $1\leq i\leq n$, define $\alpha_i=\frac{x_i}{y_{\max}}$ and let $z=\bigoplus_{i=1}^{n}\alpha_i y^{(i)}$. Then
 \[\max_{1\leq i\leq n}\alpha_i=\frac{x_{\max}}{y_{\max}}=1,\]
 which implies $\bigoplus_{i=1}^n \alpha_i=1$, so $z\in\Lambda_2$. On the other hand, for every $1\leq j\leq n$, 
 \[z_j=\max_{1\leq i\leq n} \alpha_i y^{(i)}_j=\max \left(  \frac{x_j}{y_{\max}}y_{\max},\,\max_{i\neq j}\alpha_i y_{\min}\right)=\max \left(x_j,\,\max_{i\neq j}\alpha_i y_{\min}\right).\]
 Note that for every $1\leq i\leq n$, we have $\alpha_i\leq 1$  and hence $\alpha_i y_{\min}\leq y_{\min}\leq x_{\min}\leq x_j$. Therefore,
 \[z_j=\max \left(  x_j,\,\max_{i\neq j}\alpha_i y_{\min}\right)=x_j.\]
 It follows that $x=z\in \Lambda_2$. This establishes $ \Lambda_1\subset \Lambda_2 $, which completes the proof.
\end{proof}
\end{thm}
In the examples below, we apply the preceding theorem to identify the vectors $ x $ that satisfy $ x\prec_{\max} y $ for an arbitrary vector $ y \in \mathbb{R}^n_{\max} $.
\begin{exam}
	Let $ y=\left( 2,1\right)^{T}\in\zR _{\max}^{2} $.  By using the characterization of Theorem \ref{theorem2}, we obtain
	{\fontsize{9}{.2}
	\[
	\left\{ x=\left(x_1,x_2\right)^T\in\zR _{\max}^{2} |~ x\prec_{\max} y \right\}=\left\{ x\in\zR ^{2} |~ x_1=2~\mbox {and}~1\le x_2\le 2 \right\}\bigcup
	\left\{ x\in\zR ^{2}|~ x_2=2~\mbox {and}~1\le x_1\le 2 \right\},
	\]
}
	as shown in Figure \ref{fig:1}.
	\begin{figure}[h]
		\centering
		\includegraphics[width=0.29\linewidth]{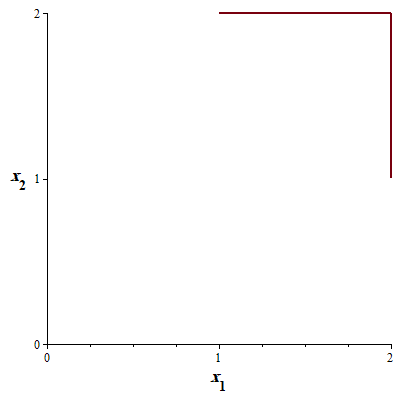}
		\caption{{\small $ \left\{ x\in\zR _{\max}^{2} |~ x\prec_{\max} \left( 2,1\right)^{T} \right\} $}}
		\label{fig:1}
	\end{figure}
\end{exam}
\begin{exam}
	Suppose $ y=(1,2,3)^T\in \zR _{\max}^{3} $. In a similar way,
\[
\left\{ x\in\zR _{\max}^{3} |~ x\prec_{\max} y \right\}=\bigcup_{i=1}^3 \left\{ x=\left(x_1,x_2,x_3\right)^T\in\zR ^{3} |~ x_i=3~\text{and for}~ j\neq i ~\text{we have}~ 1\leq x_j\leq 3\right\},
\]
as shown in Figure \ref{fig:3}.
\begin{center}
	\begin{tikzpicture}
		\begin{axis}[
			view={69}{33}, 
				axis lines=center,
				xlabel={$\bm{x_1}$}, ylabel={$\bm{x_2}$}, zlabel={$\bm{x_3}$},
				xtick={0,1,2,3,4}, ytick={0,1,2,3,4}, ztick={0,1,2,3,4},
				xmin=0, xmax=4, ymin=0, ymax=4, zmin=0, zmax=4,
				ticklabel style={font=\tiny},
			colormap/jet,
			]

			\addplot3[surf,shader=flat, draw=none, fill=red, opacity=0.5, samples=10, domain=1:3, y domain=1:3]
			({x}, {3}, {y});
			\addplot3[surf,shader=flat, draw=none, fill=red, opacity=0.5, samples=10, domain=1:3, y domain=1:3]
			({x}, {y}, {3});
			
			\addplot3[surf,shader=flat, draw=none, fill=red, opacity=0.5, samples=10, domain=1:3, y domain=1:3]
			({3}, {x}, {y});
		\end{axis}
	\end{tikzpicture}
	    \captionof{figure}{{\small $ \left\{ x\in\zR _{\max}^{3} |~ x\prec_{\max} \left( 1,2,3\right)^{T} \right\} $}}
	\label{fig:3}
\end{center}
\end{exam}
The following theorem demonstrates how max-majorization can be used to determine whether a matrix is max-doubly stochastic or not.
\begin{thm}
	The matrix \( A\in M_n\left(\zR_{\max}\right) \) is a max-doubly stochastic matrix if and only if \( A \otimes x \prec_{\max} x \) for every vector \( x \in \mathbb{R}_{\max}^{n} \).
	\begin{proof}
		Let \( A\in M_n\left(\zR_{\max}\right) \) be a matrix such that \( A \otimes x \prec_{\max} x \) for all vectors \( x \in \mathbb{R}_{\max}^{n} \). First, we show that \( A \) is a max-doubly stochastic matrix. Choose \( x = \mathbb{1} \). Then \( A \otimes \mathbb{1} \prec_{\max} \mathbb{1} \), and so by Theorem \ref{theorem2}, $A\otimes \mathbb{1}=\mathbb{1}$. Hence, \( A \) is a max-row stochastic matrix.
		Now, for \( 1 \leq i \leq n \), let \( x = e_i \). We have \( A \otimes e_i \prec_{\max} e_i \),  so the maximum entry in the \( i \)-th column of \( A \) equals \( 1 \), which shows that \( A \) is max-column stochastic matrix.
		Since \( A \) is both max-row and max-column stochastic matrix, it follows that \( A \) is a max-doubly stochastic matrix.
		
		Conversely, if \( A \) is a max-doubly stochastic matrix. It follows from the definition of max-majorization that \( A \otimes x \prec_{\max} x \) for any vector \( x \in \mathbb{R}_{\max}^{n} \) by taking \( D = A \).
	\end{proof}
\end{thm}
\section*{Acknowledgments}
I would like to express my gratitude to Dr.~Mostafa Einollahzadeh, Associate Professor in the Department of Mathematical Sciences at Isfahan University of Technology, for his invaluable guidance and support during the preparation of this article. His insightful comments and suggestions greatly enhanced the quality of this work.
\bibliographystyle{amsplain}

\end{document}